\documentclass[reqno]{amsart}
\usepackage{amssymb,amsmath,graphicx, color}

\usepackage[bookmarksnumbered,plainpages]{hyperref}

\newtheorem{thm}{Theorem}[section]
\newtheorem{lem}[thm]{Lemma}
\newtheorem{prop}[thm]{Proposition}
\newtheorem{cor}[thm]{Corollary}
\theoremstyle{defn}
\newtheorem{defn}[thm]{Definition}
\newtheorem{example}[thm]{Example}
\theoremstyle{rem}
\newtheorem{remark}[thm]{Remark}

\newenvironment{exam}{\vspace{.35cm}\begin{example}\rm ~}{\end{example}}

\numberwithin{equation}{section}
\def\beq*{\begin{eqnarray*}}
\def\eeq*{\end{eqnarray*}}
\def\bnu{\begin{enumerate}}
\def\enu{\end{enumerate}}
\def\bit{\begin{itemize}}
\def\eit{\end{itemize}}
\newcommand{\G}{\mathbb{G}}
\newcommand{\HH}{\mathbb{H}}

\newcommand{\CL}{{\rm C}_0(\G)}
\newcommand{\MG}{{\rm M}(\G)}

\newcommand{\LL}{{\rm L}^1(\G)}
\newcommand{\LTO}{{\rm L}^2(\G)}

\newcommand{\LG}{{\rm L}^\infty(\G)}

\newcommand{\N}{\mathcal{N}}
\newcommand{\C}{\mathbb{C}}
\newcommand{\LUC}{{\rm LUC}}
\newcommand{\RUC}{{\rm RUC}}
\newcommand{\HS}[1]{{#1}_{_{{\rm HS}}}}
\newcommand{\K}{\mathcal{K}}
\newcommand{\B}{\mathcal{B}}
\newcommand{\CR}{\mathcal{CR}}
\newcommand{\A}{\mathcal{A}}
\newcommand{\CB}{\mathcal{CB}}
\newcommand{\net}[1]{\{#1_i\}}

\newcommand{\tow}{\stackrel{w}{\longrightarrow}}
\newcommand{\tosw}{\stackrel{w^*}{\longrightarrow}}
\newcommand{\M}{\mathfrak{M}}
\newcommand{\NN}{\mathfrak{N}}
\newcommand{\vt}{\otimes}
\newcommand{\inj}{\check{\otimes}}
\newcommand{\proj}{\hat{\otimes}}
\newcommand{\no}[1]{\|#1\|}
\newcommand{\cbno}[1]{\|#1\|_{_{\inj}}}
\newcommand{\MX}[1]{\mathbb{M}_n(#1)}
\newcommand{\fs}{\hspace{0.3cm}}
 \newcommand{\lto}{\longrightarrow}

\begin{document}
\title[Reiter's properties  for the actions]{Reiter's properties  for the actions of locally compact quantum groups on von Neumann algebras  }
\author{M. Ramezanpour}
\address{Department of Pure Mathematics, Ferdowsi University of Mashhad, P. O. Box 1159, Mashhad 91775, Iran.}
\email{mo\_\ ra54@stu-math.um.ac.ir, md\_\ ramezanpour@yahoo.com.}

\author{H.\ R.\ Ebrahimi  Vishki}
\address{Department of Pure Mathematics and Centre of Excellence
in Analysis on Algebraic Structures (CEAAS), Ferdowsi University of Mashhad\\
P. O. Box 1159, Mashhad 91775, Iran.}
\email{vishki@ferdowsi.um.ac.ir, vishki@ams.org.}
\maketitle

%==== abstract =============%5%%%%%%%%%%%%%%%%%%%%%%%%%%%%%%%%%%%%%%%%%%%%%%%%%%%%%%%%%%%%%%%%%%%%%%%%%%%%%%%%%%%%%
\begin{abstract}The  notion of an action of a locally compact quantum group on a von Neumann algebra is studied from the amenability point of view. Various Reiter's conditions for such an action are discussed. Several applications to some  specific  actions related to certain representations and corepresentaions  are presented.
\end{abstract}

%=====  Introduction  ====%%%%%%%%%%%%%%%%%%%%%%%%%%%%%%%%%%%%%%%%%%%%%%%%%%%%%%%%%%%%%%%%%%%%%%%%%%%%%%%%%%%%%%%%%%%%
\section{\bf Introduction}

In order to extend some theories  of harmonic analysis from abelian locally compact groups (especially, to restore the so-called Pontrjagin duality theorem) to non-abelian ones, a more general object---Kac algebra, which covers all locally compact groups---constructed by some authors in early 70's; whose complete account can be found  in \cite{E-S}. The notion of quantum group introduced by V.G. Drinfeld and improved by others using an operator algebraic approach. Their approach did not satisfy some of the axioms of the Kac algebras and motivated more efforts to introduce a more general theory.  Some improvement in this direction were made by S.L. Woronowicz, S. Baaj, G. Skandalis and A. Van Daele. Finally J. Kustermans and S. Vaes, \cite{K-V, K-V_von}, introduced  the concept of locally compact quantum group along   a comprehensive set of  axioms  which covers the notion of Kac algebras, ~\cite{E-S}, and also quantum  groups. Some of the most famous locally compact quantum groups (which have  been extensively studied in abstract harmonic analysis) are   ${\rm L}^\infty(G)$ and  ${\rm VN}(G)$, in which $G$ is a locally compact group. In spite of  these two algebras have different  feature in abstract harmonic analysis but they have mostly a unified framework from the locally compact quantum group point of view. Some aspects of abstract harmonic analysis on locally compact groups are intensively  extended by V. Runde,  \cite{Run, R_uniform}, to the framework of locally compact quantum groups. The same discipline continued by the authors in  \cite{R-V}. Daws and Runde, \cite{R_ret}, introduced the Reiter's properties $P_1$ and $P_2$ for a general locally compact quantum  group. Following \cite{R_ret}, in  the present paper we come with a slightly different approach to give  various Reiter's properties  for the action of a locally compact quantum group on a von Neumann algebra. This method  generalizes not only many results of \cite{B-T, Bek, L-P, Ng, Stok} but also allows us to present the main results of \cite{R_ret} with shorter proofs.\\

Before proceeding, we fix some notations.
If $H$ is a Hilbert space, then $\B(H)$ and $\K(H)$ denote the algebra of all
bounded and compact operators on $H$, respectively. If $B$ is a $*$-algebra,
$M(B)$ denotes the multiplier algebra of $B$. As usual $\otimes$ denotes tensor product; depending on the context, it may be the algebraic tensor
product of linear spaces, the tensor product of Hilbert spaces, the minimal tensor product of C$^*$-algebras or the tensor product of von Neumann algebras. If $E$ and $F$ are operator spaces, as in \cite{E-R}, we denote the completely bounded operators from $E$ to $F$ by $\CB(E, F)$ and use $\inj$ and $\proj$ to denote the injective and projective tensor product of operator spaces, respectively; it should mentioned that in the C$^*$-algebraic setting the injective and minimal tensor products are coincide. For $n\in\mathbb{N}$, the $n$-th matrix level of $E$ is denoted by $\MX{E}$ and for a linear map $T:E\to F$, we write
$T^{(n)}:\MX{E}\to\MX{F}$ for the $n$-th amplification of $T$. If $H$ is a Hilbert space, we denote the column operator space over $H$ by $H_c$.

%================================================================================================================%
\section{{\bf Amenability of $(\NN,\alpha)$}}\label{secamen}

In this section we introduce  the amenability of $(\NN,\alpha)$, where $\alpha$ is a left action of a Hopf-von Neumann algebra $(\M,\Gamma)$ on a von Neumann algebra $\NN$. And then  we investigate some of its equivalent formulations.
 First, we formulate and briefly the notions of
 Hopf-von Neumann algebras.

\begin{defn}
A Hopf-von Neumann algebra is a pair $(\M,\Gamma)$, where $\M$ is
a von Neumann algebra and $\Gamma:\M\lto\M\vt\M$  is a  normal, unital $*$-homomorphism satisfying $(\iota\otimes\Gamma)\Gamma=(\Gamma\otimes\iota)\Gamma$.
\end{defn}

Let $(\M,\Gamma)$ be a Hopf-von Neumann algebra then the unique predual $\M_*$ of $\M$ turns into a Banach algebra under the product $*$ given by:
$$(\omega*\omega')(x)=(\omega\otimes \omega')(\Gamma(x))\quad  (\omega,\omega'\in\M_*, x\in\M).$$

Let $G$ be a locally compact group. For the so-called  Hopf-von Neumanna algebras $({\rm L}^\infty(G),\Gamma_a)$ and $({\rm VN}(G),\Gamma_s)$ we use the notations $\HH_a$ and $\HH_s$, respectively, where
\beq*
&&\Gamma_a(f)(s,t)=f(st)\fs\fs (f\in {\rm L}^\infty(G),\ s,t\in G),\\
&&\Gamma_s:{\rm VN}(G)\lto {\rm VN}(G)\vt {\rm VN}(G),\fs\fs\lambda(t)\lto\lambda(t)\otimes\lambda(t)\fs\fs(t\in G);
\eeq*
in which, $\lambda$ is  the left
regular representation of $G$ on ${\rm L}^2(G)$. It is worthwhile mentioning that for $\HH_a$ the product $*$ imposed on ${\rm L}^1(G)$  is just the usual convolution on ${\rm L}^1(G)$  whereas for $\HH_s$ it yields the usual pointwise product on ${\rm A}(G)$.

For the rest of this section we fix a Hopf-von Neumann algebra $\HH=(\M,\Gamma)$. A state $M$ of $\M$ is called
a left invariant mean for $\HH$ if
\beq*
M((\omega\otimes\iota)(\Gamma(x)))=\omega(1)M(x)\fs\fs(\omega\in\M_*,~x\in\M).
\eeq*
If there is a left invariant mean on $\HH$, we call $\HH$ left amenable.
We shall require the notion of left amenability for a more general case, as presented in the next definition.

\begin{defn}\label{leftact}
Let $\HH$ be a Hopf-von Neumann algebra, and let $\NN$ be a von
Neumann algebra. A left action of $\HH$ on  $\NN$ is a normal, injective $*$-homomorphism
$\alpha:\NN\lto\M\vt\NN$ such that $(\iota\otimes\alpha)\alpha=(\Gamma\otimes\iota)\alpha$. A state $N$ of $\NN$ is called an $\alpha$-left invariant mean if
\beq*
N((\omega\otimes\iota)(\alpha(y)))=\omega(1)N(y)\fs\fs(\omega\in\M_*,~y\in\NN).
\eeq*
If there is an $\alpha$-left invariant mean on $\NN$, we call $(\NN,\alpha)$ left amenable.
\end{defn}

\begin{defn}
Let $\HH=(\M,\Gamma)$ be a Hopf-von Neumann algebra, we say that $U$ is a unitary left corepresentation of $\HH$ and write $U\in\CR(\HH)$  if there is a Hilbert space $H_U$ such that $U$ is an unitary element of  $\M\vt\B(H_U)$ with $(\Gamma\otimes\iota)(U)=U_{13}U_{23}.$
\end{defn}

\begin{exam}\label{exam}
\bnu
\item[(i)] The comultiplication $\Gamma$ is a left action of $\HH$ on $\M$ and $\HH$ is amenable if and only if the pair $(\M,\Gamma)$ is left  amenable.
\item[(ii)] If $U\in\CR(\HH)$ then $\alpha_U:\B(H_U)\lto \M\vt\B(H_U)$, where
     $\alpha_U(x)=U^*(1\otimes x)U\ (x\in\B(H_U))$,
 is a left action of $\HH$ on $\B(H_U)$ and $(\B(H_U),\alpha_U)$ is left amenable if and only if $U$ is left amenable in the sense of \cite[Defiition 4.1]{B-T}. Moreover, $(\Gamma\otimes\iota)(U^*)=U^*_{23}U^*_{13}$ and
 $\alpha_{U^*}$ is also  a left action of $\HH$ on $\B(H_U)$. $(\B(H_U),\alpha_{U^*})$ is left amenable if and only if $U$ is right amenable in the sense of \cite[Defiition 4.1]{B-T}.
 \item[(iii)] If $G$ is a locally compact group and $\pi$ is a unitary representation of $G$ on a Hilbert space $H_\pi$,  then $\alpha_\pi:\B(H_\pi)\lto L^{\infty}(G)\vt\B(H_\pi)\simeq L^{\infty}(G,\B(H_\pi))$, where $\alpha_\pi(T)(s)=\pi(s^{-1})T\pi(s)$\  $(T\in\B(H_\pi), s\in G)$ is a left action of $\HH_a$ on the von Neumann algebra $\B(H_\pi)$. Now $(\B(H_\pi),\alpha_\pi)$ is left amenable if and only if $\pi$ is
     amenable in the sense of \cite[Definition 1.1]{Bek}.
 \item[(iv)] If $G$ is a locally compact group and $s\lto\beta_s$ is an action of $G$ on the von Neumann algebra $\NN$, then $\alpha_G:\NN\lto L^{\infty}(G)\vt\NN\simeq L^{\infty}(G,\NN)$, where
 $\alpha_G(y)(s)=\beta_{s^{-1}}(y)$\ $(y\in\NN, s\in G)$
 is a left action of $\HH_a$ on $\NN$. Moreover by \cite[Proposition I.3]{Enok}, there is a bijective correspondence  between actions of $G$ and left action of  $\HH_a$ on the von Neumann algebra $\NN$. $(\NN,\alpha_G)$ is left amenable if and only if $\NN$ is $G$-amenable in the
 sense of \cite[Definition 3.1]{L-P} or $G$ acts amenably on $\NN_*$ in the sense of \cite[Definition 1.4]{Stok}.
 \enu
 \end{exam}

We commence with the following proposition.
\begin{prop}
For a Hopf-von Neumann algebra $\HH$ the following assertions are equivalent:
\bit
\item[(i)] $\HH$ is left amenable.
\item[(ii)] For every von Neumann algebra $\NN$ and every left action $\alpha$ of $\HH$ on
it, $(\NN,\alpha)$ is left amenable.
\eit
\end{prop}
\begin{proof}
(ii)$\Rightarrow$(i) is clear. For the converse, let $M$ be a left invariant mean for $\HH$, and fix any state
$\nu\in\NN_*$ and define $N\in\NN^*$, by $N(y)=M((\iota\otimes\nu)(\alpha(y)))$\ ($y\in\NN$). For $\omega\in\M_*$ and $y\in\NN$ we have
\beq*
N((\omega\otimes\iota)\alpha(y))&=&M((\omega\otimes\iota\otimes\nu)(\iota\otimes\alpha)\alpha(y))\\
&=&M((\omega\otimes\iota\otimes\nu)(\Gamma\otimes\iota)\alpha(y))\\
&=&M((\omega\otimes\iota)\Gamma((\iota\otimes\nu)\alpha(y)))\\
&=&\omega(1)N(y),
\eeq*
that is, $N$ is an  $\alpha$-left invariant mean.
\end{proof}

For the left action $\alpha$ of $\HH$ on $\NN$, let $\alpha_*$ denote the restriction of  the adjoint of $\alpha$ to $(\M\vt\NN)_*$. Thus we have, $\alpha_*:\M_* \hat{\otimes}\NN_*\lto\NN_*$,  \cite[Corollary 4.1.9]{E-R}.The following result needs a standard arguments and so omitted.

\begin{prop}\label{FP1}
Let $\alpha$ be a left action of $\HH$ on a von Neumann algebra $\NN$, then the following assertions are equivalent:
\bit
\item[(i)] For any $\epsilon>0$ and any finite subset $\{\omega_1,\omega_2,...,\omega_n\}$ of $\M_*$, there exists a state $\nu\in\NN_*$ such that
$\no{\alpha_*(\omega_k\otimes\nu)-\omega_k(1)\nu}<\epsilon$, for $k=1,2,...,n$.
\item[(ii)] There is a net $\net{\nu}$ of state in $\NN_*$ such that
$\lim_i\no{\alpha_*(\omega\otimes\nu_i)-\omega(1)\nu_i}=0$, for all $\omega\in\M_*$.
\item[(iii)] There is a net $\net{\nu}$ of state in $\NN_*$ such that
$\alpha_*(\omega\otimes\nu_i)-\omega(1)\nu_i\tow 0$ in $\NN_*$, for all $\omega\in\M_*$.
\eit
\end{prop}

We say that a pair $(\NN,\alpha)$ has the Reiter's property $FP_1$ if it satisfies in one of the equivalent conditions of  Proposition \ref{FP1}. In the cases presented in parts (i), (iii) and (iv) of Example \ref{exam} we respectively have: $(\M,\Gamma)$ has Reiter's property $FP_1$ if and only if the condition (xi) of \cite[Theorem 2.4]{E-S.P} holds. $(\B(H_\pi),\alpha_\pi)$ has Reiter's property $FP_1$ if and only if the representation $\pi$ satisfies  in condition (iii) of \cite[Theorem 3.6]{Bek};  and $(\NN,\alpha_G)$ has Reiter's property $FP_1$ if and only if $s\to\beta_s$ satisfies  in condition (3) of \cite[Corollary 1.12]{Stok}.\\

For an action $\alpha$ of $\HH$ on a von Neumann algebra  $\NN$ we define the closed linear subspace $\LUC(\NN,\alpha)$ of $\NN$ by,
$$\LUC(\NN,\alpha)=\overline{\rm span}\{(\omega\otimes\iota)\alpha(y) ;\fs y\in\NN, \omega\in\M_*\}.$$
Since $(\omega\otimes\iota)\alpha(1)=\omega(1)1$ and  $((\omega\otimes\iota)\alpha(y))^*=(\bar{\omega}\otimes\iota)\alpha(y^*)$, where $\bar{\omega}(x)=\overline{\omega(x^*)}$,  $\LUC(\NN,\alpha)$ contains the identity and is also self adjoint.

In the cases (i), (iii) and (iv) presented in Example \ref{exam}, we respectively have: $\LUC(\M,\Gamma)=\LUC(\G)$ as defined in \cite[Definition 2.2]{R_uniform}, $\LUC(\B(H_\pi), \alpha_\pi)=X(H_\pi)$ as defined in \cite[Definition 3.1]{Bek} and $\LUC(\NN, \alpha_G)={\rm UC}(\NN)$ as defined in \cite[Definition 1.5]{Stok}.

The next result contains some fairly familiar characterizations of amenability of an action which covers  \cite[Theorem 2.4]{E-S.P}, \cite[Proposition 6.4]{B-C-T.rep}, \cite[Theorem 3.5, 3.6]{Bek} and \cite[Proposition 1.10, Corollary 1.12]{Stok} for the special cases presented in Example \ref{exam}, respectively. See also \cite[Theorem 3.4]{R_uniform}.

\begin{lem}\label{amen}
Let $\alpha$ be a left action of $\HH$ on a von Neumann algebra $\NN$. Consider the following assertions:
\bit
\item[(i)] $(\NN,\alpha)$ has Reiter's property $FP_1$.
\item[(ii)] $(\NN,\alpha)$ is left amenable.
\item[(iii)] There is a left invariant mean on $\LUC(\NN,\alpha)$.
\eit
Then $(i) \Leftrightarrow (ii) \Rightarrow (iii)$ and these are equivalent in the case where $\M_*$ has a bounded approximate identity consists of states.
\end{lem}
\begin{proof}
(i)$\Rightarrow$(ii): A direct verification reveals that any w*-cluster point of the net $\net{\nu}$, presented in Proposition \ref{FP1}, is an $\alpha$-left invariant mean.\\
(ii)$\Rightarrow$ (i): Let  $N\in\NN^*$ be an $\alpha$-left invariant mean. There exists a net $\net{\nu}$ of states in $\NN_*$ such that $\nu_i\tosw N$ in $\NN^*$, therefore  $\alpha_*(\omega\otimes\nu_i)-\omega(1)\nu_i\tow 0$ in $\NN_*$, for all $\omega\in\M_*$.\\
(ii)$\Rightarrow$(iii) is trivial. Just restrict the $\alpha$-left invariant mean to  $\LUC(\NN,\alpha)$.\\
(iii)$\Rightarrow$(ii): Let $\net{\omega}$ be a bounded approximate identity for $\M_*$ consists of states and let $M_0$ be a left invariant mean on $\LUC(\NN,\alpha)$. Define $N:\NN\lto\C$ by
 $$N(y)=\lim_{i\in I}M_0((\omega_i\otimes\iota)\alpha(y))\fs\fs(y\in\NN).$$

It is immediate that $N$ is a state on $\NN$.  Since for each $\omega\in\M_*$,\   $\lim_i(\omega\otimes\omega_i)\Gamma=\lim_i(\omega_i\otimes\omega)\Gamma$, we have
\beq*
N((\omega\otimes\iota)\alpha(y))&=&\lim_{i\in I}M_0((\omega\otimes\omega_i\otimes\iota)(\iota\otimes\alpha)\alpha(y))
=\lim_{i\in I}M_0((\omega\otimes\omega_i\otimes\iota)(\Gamma\otimes\iota)\alpha(y))\\
&=&\lim_{i\in I}M_0((\omega_i\otimes\omega\otimes\iota)(\Gamma\otimes\iota)\alpha(y))
=\lim_{i\in I}M_0((\omega\otimes\iota)\alpha((\omega_i\otimes\iota)\alpha(y)))\\
&=&\omega(1)N(y).
\eeq*
Therefore $N$ is an $\alpha$-left invariant mean on $\NN$.
\end{proof}

%%%%%%%%%%%%%%%%%%%%%%%%%%%%%%%%%%%%%%%%%%%%%%%%%%%%%%%%%%%%%%%%%%%%%%%%%%%%%%%%%%%%%%%%%%%%%%%%%%%%%%%%%%%%%%%%%%
\section{\bf Reiter's property $P_1$ for $(\NN,\alpha)$}\label{secp1}

We start this section with recalling
some of the basic definitions and properties of the von Neumann algebraic locally compact quantum groups that developed by Kustermans and Vaes in \cite{K-V, K-V_von}. We
begin with recalling some notions about normal semi-finite faithful (n.s.f.) weights on von Neumann algebras, \cite{Tak}.

Let $\M$ be a von Neumann algebra and let $\M^+$ denote its
 positive elements. For a weight $\varphi$ on $\M$ let
 $$\mathcal{M}^+_\varphi:=\{x\in\M^+ : \varphi(x)<\infty\} \fs{\rm and}\fs
 \N_\varphi:=\{x\in\M : x^*x\in\mathcal{M}^+_\varphi\}.$$

\begin{defn}
 A locally compact quantum group is a Hopf-von Neumann algebra $\HH=(\M,\Gamma)$
 such that:
 \bit
 \item there is a n.s.f. weight $\varphi$ on $\M$ which is left
 invariant, i.e.
 $$\varphi((\omega\otimes\iota)(\Gamma(x)))=\omega(1)\varphi(x)\qquad (\omega\in\M_*, x\in\mathcal{M}^+_\varphi),$$
\item there is a n.s.f. weight $\psi$ on $\M$ which is right
 invariant, i.e.
 $$\psi((\iota\otimes \omega)(\Gamma(x)))=\omega(1)\psi(x)\qquad (\omega\in\M_*, x\in\mathcal{M}^+_\psi).$$
 \eit
\end{defn}

For a locally compact group $G$, we have two locally compact quantum groups  $({\rm L}^\infty(G),\Gamma_a,\varphi_a,\psi_a)$ and $({\rm VN}(G),\Gamma_s,\varphi_s,\psi_s)$, in which $\varphi_a$ and $\psi_a$ are  the left and right Haar integrals, respectively, and  $\varphi_s=\psi_s$ is the Plancherel weight on ${\rm VN}(G)$, \cite[Definition VII.3.2]{Tak}.

Let $\M$ be in its standard form related to the GNS-construction $(H,\iota,\Lambda)$  for the left invariant n.s.f. weight $\varphi$, then there exists a unique unitary---the multiplicative unitary---$W\in B(H\otimes H)$ such that
$$W^*(\Lambda(x)\otimes\Lambda(y))=
(\Lambda\otimes\Lambda)((\Gamma(y))(x\otimes
1))\qquad (x,y\in \N_\varphi).$$
It satisfies $(\Gamma\otimes i)(W)=W_{13}W_{23}$ and $\Gamma(x)=W^*(1\otimes
x)W$\ ($x\in\M$), \cite[Theorem 1.2]{K-V_von}, \cite[p. 913]{K-V}.

Equivalently, the von Neumann algebraic quantum group $(\M,\Gamma,\varphi)$ has an underlying C$^*$-algebraic quantum group $(\A,\Gamma_c,\varphi_c)$ as discussed in \cite{K-V}, where
$$\A=\{(i\otimes\omega)(W);~~\omega\in \B(H)_*\}^{-^{\no{.}}}.$$
In which $\Gamma_c$ and $\varphi_c$ are the restriction of $\Gamma$ and $\varphi$
to $\A$ and $\A^+$, respectively, \cite[Proposition 1.7]{K-V_von}, \cite[Proposition A.2, A.5]{Van}. The dual space of $\A$ becomes a Banach algebra with the product $*_c$ given by $(f*_c g)(x)=(f\otimes g)\Gamma_c(x)$\  $(f,g\in \A^*,\ x\in \A)$. It canonically contains $\M_*$ as a closed ideal \cite[p. 193]{K-V}.

Similar to the group setting we shall use the following notation: the locally compact
quantum group $(\M,\Gamma,\varphi)$ is denoted by $\G$, and
we write $\LG$ for $\M$, $\LL$ for
$\M_*$, ${\rm L}^2(\G)$ for $H$, $\CL$ for $\A$ and $\MG$ for $\A^*$.

We shall apply notions such as left amenability and unitary left corepresentation  to locally compact quantum groups whenever they make sense for the underlying Hopf-von Neumann algebras and we write $\CR(\G)$ instead of $\CR(\HH)$. It should be mentioned that $U\in\CR(\G)$  on a Hilbert space $H_U$ is, in a sense, automatically continuous i.e. $U\in M(\CL\otimes\K(H_U))$, \cite[Theorem 1.6]{W_mult}.\\

Before we proceed for the definitions, let us describe our main aim with more details. Let $G$ be a locally compact group,  $p\in\{1,2\}$ and let $h\in {\rm L}^p(G)$. The mapping
$F[h]:G\lto {\rm L}^p(G)$ given by $F[h](s)=L_{s^{-1}}h\fs(s\in G)$
is bounded and continuous, where $(L_sh)(t)=h(st)\fs(t\in G)$. $G$ is said to  has Reiter's property $P_p$ if there is a net $\net{h}$ of non-negative, norm one functions in
 ${\rm L}^p(G)$ such that for each compact subset $K$ of $G$,
 $$\lim_i\sup_{x\in K}\no{F[h_i]-h_i}_p=0.$$ Moreover, these properties are equivalent to the amenability of $G$, \cite[Proposition 6.12]{Pier}.
Let $a\in {\rm C}_0(G)$ and define $F_a[h]:G\lto {\rm L}^p(G)$ by $F_a[h](s)=a(s)L_{s^{-1}}h \ (s\in G)$. Since $a\in {\rm C}_0(G)$ we have $F_a[h]\in {\rm C}_0(G,{\rm L}^p(G))\cong {\rm C}_0(G)\otimes^\lambda {\rm L}^p(G)$, where $\otimes^\lambda$ denotes the injective tensor product of Banach spaces. It is straightforward to verify that $G$ has Reiter's property $P_p$ if and only if there is a net $\net{h}$ of non-negative, norm one functions in
 ${\rm L}^p(G)$ such that for each $a\in {\rm C}_0(G)$,
$$\lim_i\no{F_a[h_i]-a\otimes h_i}_{\otimes^\lambda}=0.$$

Since ${\rm C}_0(G)$ is a minimal operator space, ${\rm C}_0(G)\otimes^\lambda {\rm L}^1(G)\cong {\rm C}_0(G)\inj {\rm L}^1(G)$ and
${\rm C}_0(G)\otimes^\lambda {\rm L}^2(G)\cong {\rm C}_0(G)\inj ({\rm L}^2(G))_c$, \cite[\S 8.2]{E-R}. Daws and Runde, \cite{R_ret}, by the canonical embedding ${\rm C}_0(G)\inj {\rm L}^1(G)$ and ${\rm C}_0(G)\inj ({\rm L}^2(G))_c$ into
$\CB({\rm L}^\infty(G),{\rm C}_0(G))$ and $\CB(({\rm L}^2(G))_c^*,{\rm C}_0(G))$, respectively, could extend Reiter's properties $P_1$ and $P_2$ from locally compact groups to locally compact quantum groups. Our motivation comes from the fact that ${\rm C}_0(G)\inj {\rm L}^1(G)$ and ${\rm C}_0(G)\inj ({\rm L}^2(G))_c$ can canonically embed  into
$\CB({\rm M}(G),{\rm L}^1(G))$ and $\CB({\rm M}(G),({\rm L}^2(G))_c)$, respectively.\\

Let $\alpha:\NN\lto \LG\vt\NN$ be a left action of a locally compact quantum group $\G$ on a von Neumann algebra $\NN$, and let  $a,b\in\CL$ and $\nu\in \NN_*$. We define
$$F^\alpha_{a,b}[\nu]:\MG\lto\NN^*,\fs F^\alpha_{a,b}[\nu](f)=\alpha^*(bfa\otimes\nu);\fs (f\in\MG).$$
It is immediate that $F^\alpha_{a,b}[\nu]$ is bounded.

\begin{prop}\label{Moa}
Let $\alpha$ be a left action of $\G$ on $\NN$,  $\nu\in\NN_*$, and let $a, b\in\CL$. Then $F^\alpha_{a,b}[\nu]$ lies in $\CB(\MG,\NN^*)$ and can be identified with an element of $\CL\inj \NN^*$.
\end{prop}
\begin{proof}
Let $\theta$ be a n.s.f. weight on $\NN$, and let $(\Lambda_\theta,\iota,H_\theta)$ be the GNS-construction for it. Since $\NN$ is in standard form on $H_\theta$, there are  $L,S\in\K(H_\theta)$ such that $\nu(y)=\nu(LyS)$\  $(y\in\NN)$. For $f\in\MG$ and $y\in\NN$ we have
\beq*
(F^\alpha_{a,b}[\nu](f))(y)=(bfa\otimes\nu)\alpha(y)=(f\otimes\nu)((a\otimes S)\alpha(y)(b\otimes L)).
\eeq*
Let $U_\alpha\in \LG\vt \B(H_\theta)$ be the unitary corepresentation  such that $U_\alpha^*$ is
the unitary implementation, as defined in \cite[Definition 3.6]{V_imp}, of the left action $\alpha$. Since $\alpha=\alpha_{U_\alpha}$, the mapping
$$y\to(i\otimes\nu)((a\otimes S)\alpha(y)(b\otimes L)):\NN\to\CL$$
is cb-norm limit of a net of finite rank operators in $\CB(\NN,\CL)$.
By taking adjoint, $F^\alpha_{a,b}[\nu]$ is also cb-norm limit of a net of finite rank operators in $\CB(\MG,\NN^*)$.
Thus $F^\alpha_{a,b}[\nu]\in\CB(\MG,\NN^*)$ and can be identified with an element of $\CL\inj \NN^*$.
\end{proof}

\begin{defn}\label{P1}
Let $\alpha$ be a left action of $\G$ on $\NN$, we say that $(\NN,\alpha)$ has Reiter's property $P_1$ if there is a net $\net{\nu}$ of states in $\NN_*$ such that
$$\lim_i\cbno{F^\alpha_{a,b}[\nu_i]-ab\otimes\nu_i}=0$$
in $\CL\inj \NN^*$, for all $a,b\in\CL$.
\end{defn}

In the cases (i), (iii) and (iv) presented in Example \ref{exam}, we respectively have: $(\M,\Gamma)$ has Reiter's property $P_1$ if and only if $\G$ has Reiter's property $P_1$ in the sense of \cite[Definition 3.3 ]{R_ret}; in particular $G$ has Reiter's property $P_1$ if and only if $({\rm L}^\infty(G),\Gamma_a)$ has Reiter's property $P_1$. $(\B(H_\pi),\alpha_\pi)$ has Reiter's property $P_1$ if and only if the representation $\pi$ has Reiter's property $(P_1)_\pi$ in the sense of \cite[Definition 4.1]{Bek}; and $(\NN,\alpha_G)$ has Reiter's property $P_1$ if and only if the action $s\lto\beta_s$ satisfies in condition (iii) of \cite[Proposition 3.2]{L-P} or satisfies in the condition (2) of \cite[Proposition 1.13]{Stok}.

For the proof of  the main theorem of this section that covers \cite[Theorem 4.5]{R_ret}, \cite[Proposition 3.2]{L-P}, \cite[Theorem 4.3]{Bek} and \cite[Proposition 1.13]{Stok} we quote  the following technical lemma from \cite{R_ret}.
\begin{lem}\label{unif}
({\cite[Lemma 4.4]{R_ret}}) Let $E_0, E$ and F be operator spaces, and let $S\in\CB(E,E_0)$ lies in the cb-norm closure of the finite rank operators. Then for every norm bounded net $\net{T}$ in $\CB(E_0,F)$ that
converges to $T\in\CB(E_0,F)$ pointwise on $E_0$ we have
$$\lim_i\cbno{T_i\circ S}=\lim_i\sup_{n\in\mathbb{N}}\sup_{{\ \textbf f\ }\in\MX{E}_1}\|(T_i\circ S)^{(n)}({\textbf f\ })\|_n=0;$$
where $\MX{E}_1=\{ {\ \textbf f\ }\in\MX{E}\fs;\ \no{{\ \textbf f\ }}_n\leq 1 \}$.
\end{lem}

\begin{thm}\label{thmp1}
Let $\alpha$ be a left action of $\G$ on $\NN$. Then the following assertions are equivalent:
\bit
\item[(i)] $(\NN,\alpha)$ is left amenable.
\item[(ii)] $(\NN,\alpha)$ has Reiter's property $P_1$.
\eit
\end{thm}
\begin{proof}
(i)$\Rightarrow$(ii): Let $a,b\in\CL$ and $\omega_0\in\LL$ be an arbitrary state. By Lemma \ref{amen}, there exists a net $\net{\nu'}_{i\in I}$ of states in $\NN_*$ such that $\lim_i\no{\alpha_*(\omega\otimes\nu'_i)-\omega(1)\nu'_i}=0$,  for all $\omega\in\LL$. For $i\in I$ set $\nu_i=\alpha_*(\omega_0\otimes\nu'_i)$ and define $T_i:\LL\to \NN_*$  by
\beq*
T_i(\omega)=\alpha_*(\omega\otimes\nu'_i)-\omega(1)\nu'_i.
\eeq*
Thus $\lim_i\no{\nu_i-\nu'_i}=0$ and the net $\net{T}$, which lies in $\CB(\LL,\NN_*)$, is norm bounded and also it converges to $0$ pointwise on $\LL$.

By Proposition \ref{Moa}, $S:=F^{\Gamma}_{a,b}[\omega_0]\in\CB(\MG,\LL)$  belongs to the cb-norm closure of the finite rank operators. For $f\in\MG$ we have
\beq*
F^\alpha_{a,b}[\nu_i](f)&=&\alpha_*(bfa\otimes\alpha_*(\omega_0\otimes\nu'_i))
=(bfa\otimes\omega_0\otimes\nu'_i)(i\otimes\alpha)\alpha\\
&=&(bfa\otimes\omega_0\otimes\nu'_i)(\Gamma\otimes i)\alpha
=(bfa\ast_c\omega_0\otimes\nu'_i)\alpha\\
&=&\alpha_*(S(f)\otimes\nu'_i).
\eeq*
So, by Lemma \ref{unif}, we have,
\beq*
\lim_i\cbno{F^\alpha_{a,b}[\nu_i]-ab\otimes\nu_i}
&\leq&\lim_i\Big{(}\cbno{T_i\circ S}
+\cbno{ab\otimes\nu_i-ab\otimes\nu'_i}\Big{)}=0.
\eeq*
This proves that (ii) holds.\\
(ii) $\Rightarrow$ (i): Let $\net{\nu}$ be a net satisfying  Definition \ref{P1}, and let $\omega\in\LL$. By the Cohen's factorization theorem, \cite[Corollary 2.9.26]{Dal}, there are $a, b\in\CL$ and $\omega'\in\LL$ such that $\omega=b\omega' a$. We then have:
\beq*
\lim_i\no{\alpha_*(\omega\otimes\nu_i)-\omega(1)\nu_i}&=&\lim_i\no{F^\alpha_{a,b}[\nu_i](\omega')-\omega'(ab)\nu_i}\\
&\leq&\no{\omega'}\lim_i\cbno{F^\alpha_{a,b}[\nu_i]-ab\otimes\nu_i}=0.
\eeq*
This together with Lemma \ref{amen} imply that (i) holds.
\end{proof}

Apply Theorem \ref{thmp1} together with Lemma \ref{amen} for the action presented in Example \ref{exam} (ii) we have the next result, which provides some equivalences for the left amenability of $U\in\CR(\G)$, as \cite[Defintion 4.1]{B-T}. Recall that locally compact quantum group $\G$ is called co-amenable if the Banach
algebra $\LL$ has a bounded approximate identity.
\begin{cor}
For a locally compact quantum group $\G$, consider the following assertions:
\bit
\item[(i)] $U\in\CR(\G)$ is left amenable.
\item[(ii)] $(\B(H_U),\alpha_U)$ has the Reiter's property $FP_1$.
\item[(iii)] $(\B(H_U),\alpha_U)$ has the Reiter's property $P_1$.
\item[(iv)] There is a left invariant mean on $\LUC(\B(H_U),\alpha_U)$.
\eit
Then (i) $\Leftrightarrow$ (ii) $\Leftrightarrow$ (iii) $\Rightarrow$ (iv) and these are equivalent in the case where $\G$ is co-amenable.
\end{cor}

%================================================================================================================%
\section{\bf Reiter's Property $P_2$ for $(\NN,\alpha)$}\label{secp2}

Let $\alpha:\NN\lto \LG\vt\NN$ be a left action of a locally compact quantum group $\G$ on $\NN$. Fix a n.s.f. weight $\theta$ on $\NN$ with the corresponding GNS-construction $(\Lambda_\theta,\iota,H_\theta)$. Let $U_\alpha\in \CR(\G)$,  acting on $H_\theta$, be such that $U_\alpha^*$ is the unitary implementation, as defined in \cite[Definition 3.6]{V_imp}, of the left action $\alpha$; in other words,  $\alpha=\alpha_{U_\alpha}$ on $\NN$. Let $\xi\in H_\theta$ and $a,b\in\CL$ and define
$$F^\alpha_{a,b}[\xi]:\MG\to H_\theta\fs\mbox{by}\fs F^\alpha_{a,b}[\xi](f)=(bfa\otimes i)(U_\alpha)\xi .$$
Then we have,

\begin{lem}
Let $\alpha$ be a left action of $\G$ on $\NN$,  $\xi\in H_\theta$, and let $a,b\in\CL$. Then $F^\alpha_{a,b}[\xi]$ lies in $\CB(\MG,(H_\theta)_c)$ and can be identified with an element of $\CL\inj(H_\theta)_c$.
\end{lem}
\begin{proof}
Let $L\in \K(H_\theta)$ be such that $L\xi=\xi$. Since $U_\alpha\in M(\CL\inj\K(H_\theta))$ we have
$$(i\otimes M_\xi)((a\otimes1)U_\alpha(b\otimes L))\in\CL\inj (H_\theta)_c,$$
where $M_\xi: \K(H_\theta)\to(H_\theta)_c$ is the completely bounded map given by $M_\xi(S)=S\xi$  $(S\in \K(H_\theta))$. From the canonical embedding  $\CL\inj(H_\theta)_c$ into $\CB(\MG,(H_\theta)_c)$, we have
\beq*
(f\otimes M_\xi)((a\otimes1)U_\alpha(b\otimes L))=F^\alpha_{a,b}[\xi](f).
\eeq*
In other words, $F^\alpha_{a,b}[\xi]\in\CB(\MG,(H_\theta)_c)$ and it can be identified with an element of $\CL \inj(H_\theta)_c$.
\end{proof}

Now we  define the property $P_2$ for a left action of a locally compact quantum group on
a von Neumann algebra; see also \cite[Definition 5.2]{R_ret}.
\begin{defn}\label{P2}
Let $\alpha$ be a left action of  $\G$ on $\NN$, we say that $(\NN,\alpha)$ has Reiter's property $P_2$ if there is a net $\net{\xi}$ of unit vectors in $H_\theta$ such that for every $a,b\in\CL$
$$\lim_i\cbno{F^\alpha_{a,b}[\xi_i]-ab\otimes\xi_i}=0$$
in $\CL\inj(H_\theta)_c$.
\end{defn}

It is obvious that $(\LG,\Gamma)$ has Reiter's property $P_2$ if and only if $\G$ has Reiter's property $P_2$ in the sense of \cite[Definition 5.2 ]{R_ret}.

Recall that $U\in\CR(\G)$ has the weak containment property (WCP) if there exists a net $\net{\xi}_{i\in I}$ of unit vectors in $H_U$ such that
$$\lim_i\no{U(\eta\otimes\xi'_i)-\eta\otimes\xi'_i}=0\fs\fs(\eta\in\LTO);$$
see \cite[\S 5]{B-T} for details.

Now we can prove the main result of this section that covers \cite[Theorem 5.4]{R_ret}.
\begin{thm}\label{thmp2}
Let $\alpha$ be a left action of $\G$ on $\NN$. Then the following assertions are equivalent:
\bit
\item[(i)] $U_\alpha\in\CR(\G)$ has the WCP.
\item[(ii)] $(\NN,\alpha)$ has Reiter's property $P_2$.
\eit
\end{thm}
\begin{proof}
(i) $\Rightarrow$ (ii): Let $a,b\in\CL$ and let $\omega_0\in\LL$ be an arbitrary state. Let $\net{\xi'}_{i\in I}$ be a net of unit vectors in $H_\theta$ such that
$$\lim_i\no{U_\alpha(\eta\otimes\xi'_i)-\eta\otimes\xi'_i}=0\fs\fs(\eta\in\LTO),$$
or equivalently, $\lim_i\no{(\omega\otimes i)(U_\alpha)\xi'_i-\omega(1)\xi'_i}=0$, for all $\omega\in\LL$. For $i\in I$ set $\xi_i=(\omega_0\otimes i)(U_\alpha)\xi'_i$ and define $T_i:\LL\to H_\theta$ by
\beq*
T_i(\omega)=(\omega\otimes i)(U_\alpha)\xi'_i-\omega(1)\xi'_i.
\eeq*
Then $\lim_i\no{\xi_i-\xi'_i}=0$ and the net $\net{T}$, which lies in $\CB(\LL,(H_\theta)_c)$, is norm bounded and also it  converges to $0$ pointwise on $\LL$. Let $S$ be defined as in the proof of Theorem \ref{thmp1},  then for  $f\in\MG$,
\beq*
F^\alpha_{a,b}[\xi_i](f)&=&(bfa\otimes i)(\omega_0\otimes i)(U_\alpha)\xi'_i
=(bfa\otimes\omega_0\otimes i)((U_\alpha)_{13}(U_\alpha)_{23})\xi'_i\\
&=&(bfa\otimes\omega_0\otimes i)((\Gamma\otimes i)(U_\alpha))\xi'_i
=(bfa*_c\omega_0\otimes i)(U_\alpha)\xi'_i\\
&=&(S(f)\otimes i)(U_\alpha)\xi'_i.
\eeq*
Thus, by Lemma \ref{unif}, we have
\beq*
\lim_i\cbno{F^\alpha_{a,b}[\xi_i]-ab\otimes\xi_i}
&\leq&\lim_i\Big{(}\cbno{T_i\circ S}
+\cbno{ab\otimes\xi'_i-ab\otimes\xi_i}\Big{)}=0.
\eeq*
(ii) $\Rightarrow$ (i): Let $\net{\xi}$ be a net satisfying Definition (\ref{P2}), and let $\omega\in\LL$. By the Cohen's
factorization theorem, \cite[Corollary 2.9.26]{Dal}, there are $a, b\in\CL$ and $\omega'\in\LL$ such that $\omega=b\omega' a$. Thus
\beq*
\lim_i\no{(\omega\otimes i)(U_\alpha)\xi_i-\omega(1)\xi_i}&=&\lim_i\no{F^\alpha_{a,b}[\xi_i](\omega')-\omega'(ab)\xi_i}\\
&\leq&\no{\omega'}\lim_i\cbno{F^\alpha_{a,b}[\xi_i]-ab\otimes\xi_i}=0;
\eeq*
and this completes the proof.
\end{proof}

Applying Theorem \ref{thmp2} for the action presented in Example \ref{exam} (ii) we  obtain the next result which provides an equivalence for the WCP of $U\in\CR(\G)$; see \cite[Theorem 5.1]{B-T}.
\begin{cor}
For a locally compact quantum group $\G$ the following  assertions are equivalent:
\bit
\item[(i)] $U\in\CR(\G)$ has the WCP.
\item[(ii)] $(\B(H_U),\alpha_{U})$ has Reiter's property $P_2$.
\eit
\end{cor}

In this stage we would like to remark that, it would be desirable if one could present the quantum group version of  the property $(P_2)_\pi$ for $U\in\CR(\G)$, in the sense of \cite[Definition 4.1]{Bek}.

%%%%%%%%%%%%%%%%%%%%%%%%%%%%%%%%%%%%%%%%%%%%%%%%%%%%%%%%%%%%%%%%%%%%%%%%%%%%%%%%%%%%%%%%%%%%%%%%%%%%%%%%%%%%%%%%%%%%
\section{\bf The right actions}

Since Definition \ref{leftact} and all of its consequences are given in terms of left actions, a natural question that arises is what does happen if we reset our definition and results with the right actions. Here we explain this.

Let $\HH$ and $\NN$ be a Hopf-von Neumann algebra and a von Neumann algebra, respectively.
A right action of $\HH$ on  $\NN$ is a normal, injective $*$-homomorphism
$\delta:\NN\lto\NN\vt\M$ such that $(\delta\otimes\iota)\delta=(\iota\otimes\Gamma)\delta$. A state $N$ of $\NN$ is called a $\delta$-right invariant mean if
\beq*
N((\iota\otimes\omega)(\delta(y)))=\omega(1)N(y)\fs\fs(\omega\in\M_*,~y\in\NN).
\eeq*
If there is a $\delta$-right invariant mean on $\NN$, we call $(\delta,\NN)$ right amenable. Two important example of such actions are as follows;
\bit
\item[(i)]The comultiplication $\Gamma$ is a right action of $\HH$ on $\M$ and $\HH$ is amenable if and only if the pair $(\Gamma,\M)$ is right amenable.
\item[(ii)] If $V\in\CR'(\HH)$; i.e. there is a Hilbert space $H_V$ such that $V$ is a unitary element of $\B(H_V)\vt\M$ with $(\iota\otimes\Gamma)(V)=V_{13}V_{12}$, then $\delta_V:\B(H_V)\lto \B(H_V)\vt\M$, where
     $\delta_V(x)=V^*(x\otimes 1)V\ (x\in\B(H_V))$,
 is a right action of $\HH$ on $\B(H_V)$ and $(\delta_V,\B(H_V))$ is right amenable if and only if $V$ is right amenable in the sense of \cite[Defiition 4.1]{B-T}. Moreover, $(\iota\otimes\Gamma)(V^*)=V^*_{12}V^*_{13}$ and  $\delta_{V^*}$ is also  a right action of $\HH$ on $\B(H_V)$. $(\delta_{V^*},\B(H_V))$ is right amenable if and only if $V$ is left amenable in the sense of \cite[Defiition 4.1]{B-T}.
\eit

If $\delta$ is a right action, then $\sigma\delta$ will be a left action of $\HH^{{\rm op}}=(\M,\Gamma^{{\rm op}})$ on $\NN$, where $\sigma$ denotes the flip map and $\Gamma^{{\rm op}}=\sigma\Gamma$ denotes the opposite comultiplication. So we can easily prove the right version of results presented in the left case. It should be mentioned that we must replace $\RUC(\delta,\NN)$ instead of $\LUC(\NN,\alpha)$ in the right version of Lemma \ref{amen}, where  $\RUC(\delta,\NN)$ is defined by
$$\RUC(\delta,\NN)=\overline{\rm span}\{(\iota\otimes\omega)\delta(y) ;\fs y\in\NN, \omega\in\M_*\}.$$
Since $\RUC(\Gamma,\M)=\RUC(\G)$ as defined in \cite[Definition 2.2]{R_uniform}, the right
version of Lemma \ref{amen} covers \cite[Theorem 3.4]{R_uniform} and \cite[Proposition 4.7]{Ng}.\\

Let $V\in\CR'(\G)$ act on the Hilbert space $H_V$. Set $\bar{V}=(j\otimes R)(V)$, then $\bar{V}\in\CR'(\G)$ and $H_{\bar{V}}=\bar{H}_V$, where $j:\B(H_V)\lto\B(\bar{H}_V)$ is the canonical anti-isomorphism given by $j(x)(\bar{\xi})=\overline{x^*(\xi)}$ $(\bar{\xi}\in\bar{H}_V)$ and $\bar{H}_V$ is the conjugate Hilbert space of $H_V$. Let $\mathcal{HS}(H_V)$ denote the Hilbert space of the Hilbert-Schmidt operators  on $H_V$. Let $\Theta:H_V\otimes\bar{H}_V\lto\mathcal{HS}(H_V)$  be the canonical isometric isomorphism which is given by  $\Theta(\xi\otimes\bar{\eta})(\zeta)=\langle\zeta,\eta\rangle\xi$ \ $(\xi,\zeta\in H_V, \bar{\eta}\in\bar{H}_V)$. Define a normal unital $*$-isomorphism
$$\widetilde{\Theta}:\B(H_V\otimes\bar{H}_V)\lto\B(\mathcal{HS}(H_V))\fs\fs{\rm by}\fs
\widetilde{\Theta}(T)=\Theta T\Theta^*\fs(T\in\B(H_V\otimes\bar{H}_V)),$$
also define the Hilbert-Schmidt operator $\HS{V}$  by $\HS{V}=(\widetilde{\Theta}\otimes\iota)(V\times\bar{V})$, where $V\times\bar{V}\in\CR'(\G)$ is given by $V\times \bar{V}=V_{13}\bar{V}_{23}$. Then $\HS{V}\in\CR'(\G)$ and $H_{\HS{V}}=\mathcal{HS}(H_V)$.

Now apply the right case version of Theorem \ref{thmp2} for $V\in\CR'(\G)$ to obtain the next result which provides some equivalences for the property $(\check{P}_2)$ of $V^*$, as \cite[Definition 4.2(a)]{Ng}; see \cite[Proposition 4.5]{Ng}.
\begin{cor}
Let $V\in\CR'(\G)$ act on the Hilbert space $H_V$. Then the following assertions are equivalent:
\bit
\item[(i)] $\HS{V}$ has the WCP.
\item[(ii)] $V^*$ has property $(\check{P}_2)$.
\item[(iii)] $(\delta_{\HS{V}},\B(\mathcal{HS}(H_V)))$ has Reiter's property $P_2$.
\item[(iv)] $(\delta_{(V\times\bar{V})},\B(H_V\otimes\bar{H}_V))$ has Reiter's property $P_2$.
\eit
\end{cor}
\begin{proof}
Since for every $\xi\in\mathcal{HS}(H_V)$ we have $F^{\delta_{_{V_{{\rm HS}}}}}_{a,b}[\xi]=F^{\delta_{V\times\bar{V}}}_{a,b}[\Theta^*(\xi)]$,  the proof is straightforward.
\end{proof}

%==================== References ==================================================================================%
\bibliographystyle{alpha}

\end{document}